\documentclass[reqno, 12pt]{amsart}
\usepackage{amsfonts}
\usepackage[dvips]{color}
\usepackage{amssymb,latexsym}
\usepackage[english]{babel}
\usepackage{textcomp}
\usepackage{amsmath}
\usepackage{graphicx}

\newtheorem{theorem}{Theorem}[section]
\newtheorem{proposition}[theorem]{Proposition}
\newtheorem{lemma}[theorem]{Lemma}

\newtheorem{remark}{Remark}[section]

\makeatletter
\@addtoreset{equation}{section}
\makeatother

\newcommand{\1}{\mathbb{1}}

\def\1{\,\rlap{\mbox{\small\rm 1}}\kern.15em 1}

\def\build#1_#2^#3{\mathrel{\mathop{\kern 0pt#1}\limits_{#2}^{#3}}}
\def\tend#1#2{\build\hbox to 12mm{\rightarrowfill}_{#1\rightarrow #2}^{ }}

\def\converge#1#2#3#4{\build\hbox to
#1mm{\rightarrowfill}_{#2\rightarrow #3}^{\hbox{\scriptsize #4}}}

\newcommand{\beq}{\begin{equation}}
\newcommand{\eeq}{\end{equation}}

\bibdata{prob}
\bibstyle{alpha}

\title{An important excited random walk counter-example}

\author{Rafael Santos$^1$}
\thanks{1. Supported by Capes with a scholarship}

\address{
\newline
Rafael Santos
\newline
Universidade Federal do Rio de Janeiro, Instituto de Matem\'atica, secretaria da p\'os-gradua\c{c}\~ao.
\newline  Caixa Postal 68530, 21945-970, Rio de Janeiro, Brasil.
\newline
e-mail: {\rm \texttt{rafaels@dme.ufrj.br}}
}


\begin{document}

\maketitle

\vspace{-0.3cm}

\begin{abstract}
In this paper, we give a detailed construction of an example of excited random walk with speed zero in an ergodic random environment that have an infinite average number of cookies in each site. This example confirms that a result of Mountford, Pimentel and Valle ($2006$), which gives a sufficient condition for excited random walks in deterministic environment to have positive speed, can not be extended to ergodic random environment.\
\end{abstract}

\section{Introduction}

The excited random walk (ERW) in $\mathbb{Z}$ is a non Markovian random walk introduced by Zerner \cite{Zerner2004} that can be informally described in the following way: Initially to each vertex $z \in \mathbb{Z}$ we associate a non-negative number $M_z$ and will say that there are $M_z$ cookies on $z$; these cookies can change the jump probability of the particle. After the environment is settled, a particle will begin to move from an initial vertex $y_0 \in \mathbb{Z}$ and in each jump, it will choose one of the two neighbor positions according to the following rule: If the particle is in a vertex that doesn't have cookies anymore, it will jump to the right with probability $\frac{1}{2}$. Otherwise, if the vertex has at least one cookie, the particle will consume a cookie and jump to the right with fixed probability $p > \frac{1}{2}$.

Now for a formal description, consider an environment $\omega$ as an element of 
 
\vspace{-0.5cm}

$$\Omega = \left\{ (\omega (z,i))_{z \in \mathbb{Z},i \in \mathbb{N}} \mbox{  } | \mbox{  } \omega(z,i) \in [0,1], \forall i \in \mathbb{N} \mbox{ \ e \ } \forall z \in \mathbb{Z} \right\}.$$\

\vspace{-0.4cm}

The value of $\omega(z,i)$ gives the probability that the ERW will jump from $z$ to $z+1$ when it visits the state $z$ for the $i-th$ time. Fixing $\omega \in \Omega$ and $y_0 \in \mathbb{Z}$, we have that an ERW $(Y_n)_{n \geq 0}$ starting from $y_0$ in an environment $\omega$ is a stochastic processes with probability measure $P_{y_0, \omega}$ satisfying:\

\vspace{-0.4cm}

$$P_{y_0, \omega} [Y_0 = y_0] = 1,$$

\vspace{-0.6cm}

$$P_{y_0, \omega} [Y_{n+1} = Y_n + 1 | (Y_i)_{0 \leq i \leq n}] = \omega (Y_n, \# \{ i \leq n : Y_i = Y_n \}),$$

\vspace{-0.6cm}

$$P_{y_0, \omega} [Y_{n+1} = Y_n - 1 | (Y_i)_{0 \leq i \leq n}] = 1 - \omega (Y_n, \# \{i \leq n : Y_i = Y_n \}).$$\

\vspace{-0.4cm}

The environment $\omega$ can be initially fixed or randomly determined according to a probability measure in $\Omega$. We will denote by $\omega^{M,p}$ the homogeneous environment with an initial amount of $M$ cookies in each site and jump probability $p$, i.e.\

\vspace{-0.5cm}

$$\omega(x,k) = \begin{cases} p, \mbox{ if } k \leq M  \mbox{,} \\ \vspace{-0.5cm} \\ \frac{1}{2}, \mbox{ if } k > M. \end{cases}$$\

\vspace{-0.1cm}

In \cite{Zerner2004} it was obtained a Law of Large Numbers for the ERW in general environments, proving that the walk's speed, given by\ 

\vspace{-0.3cm}

$$\displaystyle{ a.s.- \lim_{n \rightarrow \infty} \frac{Y_n}{n}},$$\

\vspace{-0.1cm}

\noindent exists and that, for $p<1$ and $M=2$, the ERW has speed zero. After that, Mountford, Pimentel and Valle \cite{GlaucoEx} proved the following theorem (Theorem 1.1 of that article):

\vspace{0.1cm}

\begin{theorem}
\label{teo_glauco}

For the ERW starting at $0$ in the environment $\omega^{M,p}$, we have that:\

\vspace{0.1cm}

(i) For every $p \in \left( \frac{1}{2},1 \right)$, there exists $M_0 = M_0(p)$ sufficiently large such that the walk's speed is positive for all $M > M_0$.\

\vspace{0.1cm}

(ii) If $p$ and $M$ satisfy $M(2p-1) \in (1,2)$, then ERW is transient, but with speed zero.\

\end{theorem} 

\vspace{0.1cm}

The authors remarked that Theorem \ref{teo_glauco} (i) cannot be generalized to ergodic random environment with an average number of cookies per site greater than $M$ (see remark \ref{remark1} below) and they gave a very brief description of a counter-example construction to verify that (without much explanation, since this remark was not one of the main objectives of that article). But considering the complexity of the counter-example, we think that is hard for the reader to fully understand how it is constructed and why it really works for this purpose. So, our objective here is to give a very detailed construction of an ergodic environment where the average number of cookies in each site is infinite and the ERW has zero speed.\

\vspace{0.1cm}

\begin{remark}
\label{remark1}
In Kosygina and Zerner \cite{Zerner2008}, Theorem \ref{teo_glauco} is extended, giving sufficient conditions for ERW in random environments to have positive speed (Theorem 2 of that article).\

\end{remark}

\vspace{0.2cm}

More details and results related to ERW's can be found in \cite{ZernerSurvey}.


\section{Construction of the counter-example}

Let's define an ERW $(Y_n)_{n\geq 0}$ in the following way: Fix $(Z_j)_{j\geq 1}$ and $(Z^{-}_j)_{j\geq 1}$ sequences of independent and identically distributed random variables such that:\

\vspace{-0.4cm}

$$P(Z^{-}_j = -(2^n)) = P(Z_j=2^n) = \frac{\gamma}{4^{\epsilon n}},$$

\vspace{0.1cm}

\noindent for all $n\geq 2$, where $\gamma = \left(\displaystyle{\sum^{\infty}_{n=2}}\frac{1}{4^{\epsilon n}} \right)^{-1}$ and $\epsilon > 0$ is a constant that will be fixed later.\


Note that, if $\epsilon > 0.5$, the random variables $Z_j`s$ have finite first moment. From now on we will consider $\frac{1}{2} < \epsilon < 1$.\


Now fix $p \in \left( \frac{1}{2} , 1 \right)$ and consider the following environment:\

\vspace{-0.2cm}

$$\tilde{\omega}(x,j) = \begin{cases} \frac{1}{2}, \mbox{if } x \in \left[\displaystyle{\sum^{l}_{i=1}}Z_i + 1,\displaystyle{\sum^{l+1}_{i=1}}Z_i - 1 \right] \mbox{or } x \in \left[\displaystyle{\sum^{l}_{i=1}}Z^{-}_i - 1,\displaystyle{\sum^{l+1}_{i=1}}Z^{-}_i + 1 \right] \mbox{,} \\ \ \ \ \vspace{-0.4cm} \\ p, \mbox{ if } x = \displaystyle{\sum^{l}_{i=1}}Z_i \mbox{ or }  x = \displaystyle{\sum^{l}_{i=1}}Z^{-}_i  \mbox{, for any $l$ and $j$}. \end{cases}$$\

That is, for each value of $Z_j$ and $Z^{-}_j$, we have an interval of size $Z_j - 1$ without cookies; and in each vertex between them, we put infinite cookies. In figure \ref{paeex} we have the ilustration of the environment $\tilde{\omega}$\

\vspace{0.1cm}

\begin{figure}[h]

\centering

\includegraphics[scale=0.48]{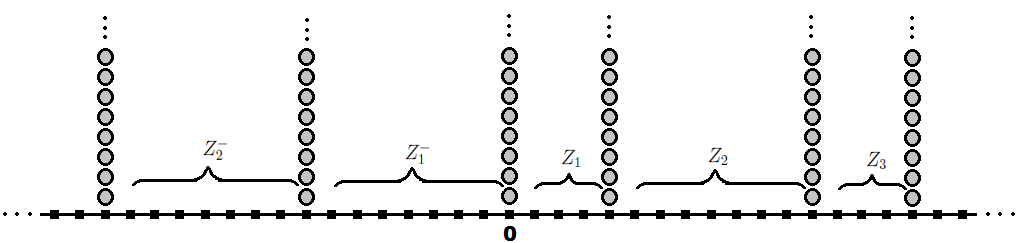}

\caption{ilustration of the environment $\tilde{\omega}$, where we have intervals of size $Z_j - 1$ without cookies and infinite cookies in their border.}

\label{paeex}

\end{figure}

The environment $\tilde{\omega}$ is not ergodic, but now we will construct an ergodic environment based on $\tilde{\omega}$, which we denote by $\omega^{*}$.

\vspace{-0.2cm}

$$ {\omega^{*}} (x,i) = \tilde{\omega}(x-U,i),$$\

\vspace{-0.2cm}

\noindent where $U|Z^{-}_1$ has uniform distribution in $\{ 0,1,2,\ldots,Z^{-}_1 - 1 \}$.\

It means that our environment $\omega^{*}$ is obtained by just take $\tilde{\omega}$ and translate the positions with infinite cookies by $U$ to the left, where $U$ is drawn uniformly among the integers between $0$ and $Z^{-}_1-1$. Since the expected value of $Z_j$ and $Z^{-}_j$ are both finite, we have that $E[M_z] = \infty$ for all $z \in \mathbb{Z}$ in environment $\omega^{*}$.\

Consider that $(Y_n)_{n\geq 0}$ is the ERW associated to $\omega^{*}$. We will show that $Y_n$ has zero speed almost surely. To simplify the notation, we will assume during the proof that we have $U|Z^{-}_1 = 0$ (which is equivalent to say that we will consider $\omega^{*}$ = $\tilde{\omega}$), but all arguments that we will use during the proof clearly works for any possible value of $U|Z^{-}_1$. This gives the example that we claimed to obtain.\

\vspace{0.3cm}

Define $T_K = inf\{n\geq0 : Y_n = K \}$, the first time that the walk reach the position $K$. To prove that this process has zero speed we will show that $\displaystyle{\limsup_{K \rightarrow \infty} \frac{T_K}{K}}=\infty$ almost sure. This is a standard approach and is the same used in \cite{Zerner2004} to prove the existence of $\displaystyle{\lim_{n \rightarrow \infty} \frac{Y_n}{n}}$ for ERW (Theorem 13 of that article).\

Now, since the variables $T_K`s$ are a.s. finite, it's enough to show that the following proposition holds:\

\vspace{0.1cm}

\begin{proposition}
\label{prop_ap}

For all $n \in \mathbb{N}$, $P \left( \displaystyle{\lim_{K \rightarrow \infty} \frac{T_K}{K}} \geq c4^{(1-\epsilon)n} \right) =1$, where $c > 0$ is a constant that doesn't depend of $n$ and $K$.\

\end{proposition}

\vspace{0.3cm}

To prove this proposition we will first state and prove two lemmas.\

\vspace{0.3cm}

\begin{lemma}
\label{lema_ap1}

Conditioned to $Z_l \geq 2^n$, the time spent by $(Y_n)_{n\geq 0}$ to cross an interval $\left[\displaystyle{\sum^{l}_{i=1}}Z_i + 1,\displaystyle{\sum^{l+1}_{i=1}}Z_i - 1 \right]$ is at least $4^n$, with probability higher than a constant $\beta > 0$, not depending on $n$.\

\end{lemma}

\vspace{0.01cm}

\begin{proof}

Since we don't have cookies in the interval given at the statement, $Y_n$ will have the same behavior of a symmetric simple random walk. \

Denote by $X_n$ the symmetric simple random walk with reflection in the origin and let $(X_n, Y_n)$ be the coupling of $X_n$ with $Y_n$ after it enters inside an interval of size $2^n-1$ without cookies. That is, these two process will evolve together, except if $(X_n, Y_n)$ return to the left boundary of the interval. In this case $X_n$ jumps to the right with probability 1 and $Y_n$ with probability $p$, and then, at this point, these two random walks can split. If it happens, $X_n$ and $Y_n$ will move independently until they meet each other again.\

Note that $Y_n$ will be always in the same position of $X_n$ or at its left, and consequently, the probability of the event $\{Y_n$ requires a time greater than $4^n$ to cross the interval of size $2^n-1$ without cookies$\}$ is limited from below by the probability of the event $\{X_n$ reach the position $2^n$ for the first time in $t>4^n\}$. By Donsker invariance principle (see, e.g., section 8.6 of \cite{Durret}), this probability will converge to the probability of the event $\{$A standard Brownian motion with total reflection in origin hit the value $1$ for the first time after a time greater than $1\}$.\

Denoting by $B_t$ the position of a standard Brownian motion at time $t$ and letting $\tilde{T_a} = min \{t ; \left| B_t \right| \geq a \}$, standard operations gives us that $P(\tilde{T_{1}} > 1) = 1 - \Phi(1)$, where $\Phi(\bullet)$ is the distribution function of a standard normal. So, there exists $N_0 \in \mathbb{N}$ such that $P(T_{2^n} \geq 4^n) \geq 0.3$ for all $n \geq N_0$. Beside that, for $n < N_0$ we have that $\displaystyle{ \min_{2 \leq n < N_0} P(T_{2^n} \geq 4^n)} > (1-p)^{4N_0} > 0$.\







To conclude, we can fix $\beta = \min \displaystyle{ \left( 0.3 \ , \min_{2 \leq n < N_0} P(T_{2^n} \geq 4^n) \right)}$ and then we have that the probability that the ERW $(Y_n)_{n\geq 0}$ will require a time greater than $4^n$ to cross an interval of size $2^n-1$ without cookies is at least $\beta$.

\end{proof}

\begin{lemma}
\label{lema_ap2}

Denote by $N_n(K)$ the number of intervals with size $2^n-1$ without cookies and with right boundary in $[0,K]$, that is,

\vspace{-0.2cm}

$$N_n(K) = \# \left\{l: \displaystyle{\sum^{l}_{i=1}}Z_i < K \mbox{, } Z_l = 2^n \right\}.$$


\noindent There exists a constant $\alpha > 0$ such that:\

\vspace{-0.1cm}

$$P \left( \displaystyle{\liminf_{K \rightarrow \infty} \frac{N_n(K)}{K}} \geq \frac {\alpha}{4^{\epsilon n}}  \right)=1, \ \forall n \in \mathbb{N}.$$\

\end{lemma}

\begin{proof}

\vspace{0.1cm}

Let $N(K)$ Be the number of sites with infinite cookies contained in $[0,K]$, that is, $N(K)=\# \left\{l: \displaystyle{\sum^{l}_{i=1}}Z_i < K \right\}$.\

\vspace{0.1cm}

Note that $N(K)$ is a renewal process where the time between renewals  is given by the $Z_j's$. So, applying the Law of Large Numbers to renewal processes, we have:\

\vspace{-0.2cm}

$$P \left( \displaystyle{\lim_{K \rightarrow \infty}\frac{N(K)}{K}} = \frac{1}{E(Z_j)} \right) = 1.$$\

\vspace{0.1cm}

besides that:

$$ \left\{ \displaystyle{\lim_{K \rightarrow \infty}\frac{N(K)}{K}} = \frac{1}{E(Z_j)} \right\} \subset \bigcup_{K_0} \left\{ N(K) > \frac{K}{2E[Z_j]}, \hspace{0.1cm} \forall K > K_0 \right\}.$$\


Since we have an increasing union of events, it follows that:

$$\displaystyle{\lim_{K_0 \rightarrow \infty} P \left( N(K) > \frac{K}{2E[Z_j]}, \hspace{0.1cm} \forall K > K_0 \right)} = 1.$$\

\vspace{0.2cm}

Consequently, for all $K > K_0$, we have:

$$P \left( \displaystyle{\liminf_{K \rightarrow \infty} \frac{N_n(K)}{K} \frac{1}{2E[Z_j]}} \geq P(Z_j = 2^n) \right) \geq $$

\vspace{0.2cm}

$$ \geq P \left( \displaystyle{\liminf_{K \rightarrow \infty}\frac{N_n(K)}{N(K)}} \geq P(Z_j = 2^n) \right) \geq $$

\vspace{0.2cm}

$$ \geq P \left( \displaystyle{\liminf_{K \rightarrow \infty}\frac{N_n(K)}{N(K)}} \geq P(Z_j = 2^n) \left| N(K) > \frac{K}{2E[Z_j]}, \forall K > K_0 \right. \right) \times $$ 

\vspace{-0.2cm}

$$ \times P \left(N(K) > \frac{K}{2E[Z_j]}, \forall K > K_0 \right).$$\

When $K_0 \rightarrow \infty$ we have by the Law of Large Numbers that this last product of probabilities will be approximately:

\vspace{-0.1cm}

$$P \left( \displaystyle{\lim_{N \rightarrow \infty}\frac{1}{N} \sum_{j=1}^{N} I_{\{Z_j = 2^n\}}} = P(Z_j = 2^n) \right) = 1.$$\

Therefore, for all $K > K_0$,

\vspace{-0.1cm}

$$N_n(K) \geq P(Z_j = 2^n) \frac{K}{2E(Z_j)} = \frac{\gamma}{4^{\epsilon n}} \frac{K}{2E(Z_j)}.$$\

To conclude, we just need to take $\alpha = \frac{\gamma}{2E(Z_j)}$ to satisfy the lemma's inequality.
\end{proof}

\vspace{0.4cm}

Now we will prove the Proposition \ref{prop_ap}:\\

\vspace{-0.2cm}

\begin{proof}

Fix $n \geq 2$. Almost surely there exists a $K_0 \in \mathbb{N}$ such that\

\vspace{-0.2cm}

$$\left\{ \displaystyle{\liminf_{K \rightarrow \infty} \frac{N_n(K)}{K}} \geq \frac {\alpha}{4^{\epsilon n}}  \right\} \subset \left\{ \frac{N_n(K)}{K} \geq \frac {\alpha 4^{- \epsilon n}}{2}, \  \forall K > K_0 \right\}.$$\

Now we can write:\

\vspace{-0.2cm}

$$P \left( \displaystyle{\sup_{K} \frac{T_K}{K}} \geq c4^{(1-\epsilon)n} \right) \geq P \left( \displaystyle{\sup_{K} \frac{T_K}{K}} \geq c4^{(1-\epsilon)n}, \frac{N_n(K_0)}{K_0} \geq \frac {\alpha 4^{- \epsilon n}}{2}  \right) =$$


$$= P \left( \displaystyle{\sup_{K} \frac{T_K}{K}} \geq c4^{(1-\epsilon)n} \left| N_n(\tilde{K}) \geq \frac {\alpha 4^{- \epsilon n}}{2} \tilde{K} \right. \right) P \left( N_n(\tilde{K}) \geq \frac {\alpha 4^{- \epsilon n}}{2} \tilde{K} \right).$$\

From Lemma \ref{lema_ap2}, we have that the probability in the rightmost term of the product above converge to $1$ when $\tilde{K} \rightarrow \infty$. To analyze the first probability, let's consider the following event:\

\vspace{0.2cm}

$E_k = \left\{ \right.$For at least $\frac{cK}{4^{\epsilon n}}$ intervals of size $2^n$ among $\left\lceil \frac{\alpha K 4^{- \epsilon n}}{2}\right\rceil$ available intervals we have a crossing time greater than $4^n \}$.\

\vspace{0.2cm}

Now we have:\


$$P \left( \displaystyle{\sup_{K} \frac{T_K}{K}} \geq c4^{(1-\epsilon)n} \left| N_n(K_0) \geq \frac {\alpha 4^{- \epsilon n}}{2} K_0 \right. \right) \geq P(E_{K_0}).$$\

\vspace{0.1cm}

Note that, conditioned to the choice of the environment, we have that $Y_n$ is a non-homogeneous Markov process, indeed the stock of cookies in each site is constant zero or infinite; and then the jump probabilities do not depend on the history of the process. So, by the markov property and recalling from Lemma \ref{lema_ap1} that an interval of size $2^n-1$ without cookies will require a time greater than $4^n$ with a probability higher than $\beta$, we have:\


$$P(E_{K_0}) \geq P \left( Bin \left( \left\lceil  \frac {\alpha K_0 4^{- \epsilon n}}{2} \right\rceil, \beta \right) \geq \frac{c K_0}{4^{\epsilon n}} \right).$$\


Since $K_0$ will be taken large enough, the first parameter of this binomial will be at least one. To conclude, note that, by the Law of Large Numbers, this probability converge to $1$ when $K_0 \rightarrow \infty$, if the expected value of this binomial is greater than $\frac{c K_0}{4^{\epsilon n}}$. To achieve that, we just need to take $c < \frac{\alpha \beta}{2}$. 
\end{proof}

\vspace{0.1cm}

\medskip \smallskip

\noindent{\bf Acknowledgments:} This work was developed by the author as part of his master thesis at Universidade Federal do Rio de Janeiro (UFRJ). So, the author thanks a lot his advisor, Glauco Valle (UFRJ), for all the support during the realization of this work. He also thanks the committee members, Leandro Pimentel (UFRJ) and Remy Sanchis (UFMG), for giving some contribution to improve this work.\\


\medskip \smallskip

\begin{thebibliography}{99}

\bibitem{Durret} R. DURRET. \textit{Probability: theory and examples.} Cambridge Series in Statistical and Probabilistic Mathematics. Cambridge University Press, Cambridge, fourth edition (2010).

\bibitem{Zerner2008} E. KOSYGINA and M. P. W. ZERNER. Positively and negatively excited random walks on integers, with branching processes. \textit{Electron}, \textbf{Vol.}13(64):1952-1979 (2008).

\bibitem{ZernerSurvey} E. KOSYGINA and M. P. W. ZERNER. Excited random walks: Results, methods, open problems. \textit{Bull. Inst. Math. Acad. Sin.}, \textbf{Vol.}8(1):105-157 (2013).

\bibitem{GlaucoEx} T. MOUNTFORD, L. P. R. PIMENTEL, G. VALLE. On the speed of the one-dimensional excited random walk in the transient regime. \textit{ALEA}, \textbf{Vol.}2:279-296 (2006).

\bibitem{Zerner2004} M. P. W. ZERNER. Multi-excited random walks on integers. \textit{Probab. Th. and Rel. Fields}, \textbf{Vol.}133(1) (2005).

\end{thebibliography}
\end{document}